\documentclass[reqno]{amsart}

\usepackage{graphicx,color}
\usepackage{amsmath}

\numberwithin{equation}{section}
\newtheorem{theorem}{Theorem}[section]
\newtheorem{lemma}[theorem]{Lemma}

\newtheorem{remark}[theorem]{Remark}

\newtheorem{claim}{Claim}

\newcommand{\R}{\mathbb{R}}
\newcommand{\dint}{\displaystyle\int}

\begin{document}

\title[Schr\"odinger equation asymptotically periodic]{Schr\"odinger equations with asymptotically periodic terms}

\author[Marchi, R.]{\textbf{Marchi, Reinaldo de}\\
 Departamento de Matem\'atica, Universidade Federal de Mato Grosso,Cuiab\'a 78060-900, MT, Brazil
(reinaldodemarchi@ufmt.br)}
\label{firstpage}
\maketitle

\begin{abstract}
We study the existence of nontrivial solutions for a class of asymptotically periodic 
semilinear Schr\"odinger equations in $\R^N$. By combining variational methods and the concentration-compactness principle we obtain a nontrivial solution for asymptotically 
periodic problem and a ground state solution for the periodic problem. In the proofs we apply the Mountain Pass Theorem and its local version.
\end{abstract}


\section{Introduction}

In this article, we study the existence of nontrivial solutions
for the semilinear Schr\"odinger equation
\begin{equation}\label{eq1}
-\Delta u+V(x)u=f(x,u),\ x\in\R^N,
\end{equation}
where $V:\R^N\to\R$ and $f:\R^N\times\R\to\R$ are continuous
functions. In our main result we establish the existence of
solution for the problem (\ref{eq1}) under an asymptotic
periodicity condition at infinity.

In order to precisely state our results we denote by $\mathcal{F}$
the class of functions $h\in C(\R^N,\R)\cap
L^\infty(\R^N,\mathbb{R})$ such that, for every $\varepsilon>0$,
the set $\{x\in\R^N: |h(x)|\geq\varepsilon\}$ has finite Lebesgue
measure. We suppose that $V$ is a perturbation of a periodic
function at the infinity in the following sense:

\begin{itemize}
\item[$(V)$]  there exist a constant $a_0>0$ and a function $V_0\in
C(\R^N,\R)$, 1-periodic in $x_i$, $1\leq i\leq N$, such that
$V_0-V\in\mathcal{F}$ and
$$V_0(x)\geq V(x)\geq a_0>0,\ \text{for all}\ x\in\R^N.$$
\end{itemize}

Considering $F(x,t)=\int_0^t f(x,s)ds$ the primitive of $f\in
C(\R^N\times\R,\R)$, we also suppose the following hypotheses:
\begin{itemize}
\item[$(f_1)$] $F(x,t)\geq0$ for all $(x,t)\in\R^N\times\R$ and $f(x,t)=o(t)$ uniformly in $x\in\R^N$ as $t\to0$;
\item[$(f_2)$] there exists a function $b\in C(\R\setminus\{0\},\R^+)$ such that
$$
\widehat{F}(x,t) := \frac{1}{2} f(x,t)t - F(x,t) \geq b(t)t^2,
$$
for all $(x,t)\in\R^N\times\R$;
\item[$(f_3)$] there exist $a_1>0,\,R_1>0$ and $\tau>\max\{1,N/2\}$ such that
$$|f(x,t)|^\tau\leq a_1|t|^\tau\widehat{F}(x,t),
$$
for all $(x,t)$ with $|t|>R_1$ ;
\item[$(f_4)$] uniformly in $x\in\R^N$ it holds
$$
\lim_{|t| \to +\infty} \frac{F(x,t)}{t^2}= +\infty ;
$$
\item[$(f_5)$] there exist $q\in(2,2^*)$ and functions $h\in\mathcal{F}$, $f_0\in C(\R^N\times\R,\R)$, 1-periodic in $x_i$, $1\leq i\leq N$, such that:
\begin{itemize}
\item[(i)] $F(x,t)\geq F_0(x,t)=\int_0^t f_0(x,s)ds$, for all $(x,t)\in\R^N\times\R$;
\item[(ii)] $|f(x,t)-f_0(x,t)|\leq h(x)|t|^{q-1}$,  for all $(x,t)\in\R^N\times\R$;
\item[(iii)] $\frac{f_0(x,\cdot)}{|\cdot|}$ is increasing in $\mathbb{R} \setminus \{0\}$, for all $x \in \mathbb{R}^N$.
\end{itemize}
\end{itemize}

The main result of this paper can be stated as follows:

\begin{theorem}\label{th1}
Suppose that $V$ and $f$ satisfy $(V)$ and $(f_1)-(f_5)$,
respectively. Then the problem $(\ref{eq1})$ possesses a solution.
\end{theorem}

As a by product of our calculations we can obtain a weak solution
for the periodic problem. In this setting we can drop the
condition $(f_5)$, and we shall prove the following result:

\begin{theorem}\label{th2}
Suppose that $V(\cdot)$ and $f(\cdot,t)$ are 1-periodic in $x_i$,
$1 \leq i \leq N$, and $V(x) \geq a_0>0$ for all $x \in
\mathbb{R}^N$. If $f$ satisfies $(f_1)$, $(f_3)$, $(f_4)$ and
\begin{itemize}
\item[$({f_2})'$] $\widehat{F}(x,t)>0$ for all $t\neq0$,
\end{itemize}
then the problem $(\ref{eq1})$ possesses a ground states solution.
\end{theorem}

Problems as $(\ref{eq1})$ has been focus of intensive research in
recent years. Initially, several authors have dealt with the case
where $f$ behaves like $q(x)|u| ^{p-1}u, 1<p<2^*-1$ and $V$ is
constant (see ~\cite{berestycki1983nonlinear,bahri1990min}). In the work of Rabinowitz ~\cite{rabinowitz1992class}
and Rabinowitz-Coti Zelati ~\cite{zelati1992homoclinic} it was imposed the classical
superlinear condition due to Ambrosetti-Rabinowitz:

\begin{itemize}
\item[$(AR)$] there exists $\mu>2$ such that
$$
0<\mu F(x,t)\leq f(x,t)t
$$
for all $x\in\R^N$ and $t\neq0$.
\end{itemize}
This hypothesis has a important role to show that (PS) sequences
are bounded. In this work we assume the condition $(f_4)$ which is
weaker than the condition of $(AR)$. It has already appeared in
the papers of Ding-Lee ~\cite{ding2006multiple} and Ding-Szulkin
~\cite{ding2007bound}.

We emphasize that, in Theorem \ref{th1}, we are not supposing
periodicity on $V$ or $f(\cdot,t)$. Instead, we consider the
asymptotically periodic case as done in the paper of Lins-Silva
\cite{lins2009quasilinear}. The condition $(f_5)$ describes our
assumption of asymptotically periodic for the nonlinearity $f$. A
pioneering work on problems as $(\ref{eq1})$ is due Alama-Li
~\cite{alama1992multibump} that focused the case $V\equiv1$ and
$f$ asymptotically periodic in a weaker sense. We also cite the
papers \cite{alves2001perturbations, alves2001nonlinear,
lins2009quasilinear, silva2010quasilinear, silva2010quasilinear2}
for some related (and not comparable) results.

As an example of application of our main theorem we take $a\in
C(\R^N,\R)\cap L^{\infty}(\mathbb{R}^N,\mathbb{R})$ 1-periodic in
$x_i$, $1\leq i\leq N$ with $a(x)\geq2$. Define the functions
$$f(x,t)=a(x)t\ln(1+t)+e^{-|x|^2}t(\ln(1+t)+1-\cos(t)),\ t\geq0,$$
$$f_0(x,t)=a(x)t\ln(1+t),\ t\geq0,$$
and $f(x,t)=-f(x,-t)$, $f_0(x,t)=-f_0(x,-t)$ for $t<0$. This
function satisfies $(f_1)-(f_5)$, but not satisfies $(AR)$.
Moreover $f(x,t)/t$ is oscillatory, and therefore the Nehari
approach used in \cite{szulkin2010method}  is not applicable.

The rest of the article is organized as follows. In Section 2 we
present the technical results that be used throughout the work.
The final Section 3 is devoted to the proof of Theorems \ref{th1}
and \ref{th2}.

\section{Preliminary Results}

In this section we present some preliminaries for the proofs of
our main theorems. We denote by $B_R(y)$ the open ball in $\R^N$
of radius $R>0$ and center at the point $y$. The Lebesgue measure
of a set $A\subset\R^N$ will be denoted by $|A|$. To shorten
notation, write  $\int_A u$ instead of $\int_A u(x)dx$. We also
omit the set $A$ whenever $A=\R^N$. We write $|\cdot|_p$ for the
norm in $L^p(\R^N)$.

Throughout the paper we assume that the potential $V$ satisfies
the assumption $(V)$. This implies that the norm
$$\|u\|^2=\int(|\nabla u|^2+V(x)u^2),~~~u \in H^1(\R^N)
$$
is equivalent to the usual one. In what follows we denote by $H$
the space $H^1(\mathbb{R}^N)$ endowed with the above norm.

In our first lemma we obtain the basic estimates
on the behavior of the nonlinearity $f$.

\begin{lemma}\label{lemma1.1}
Suppose that $f$ satisfies $(f_1)$, $(f_3)$ and $(f_5)(ii)$. Then,
for any given $\varepsilon>0$, there exists $C_\varepsilon>0$ and
$p \in (2,2^*)$ such that
\begin{equation} \label{c1}
|f(x,t)|\leq\varepsilon|t|+C_\varepsilon|t|^{p-1},\,\,\,\,
|F(x,t)|\leq\varepsilon|t|^2+C_\varepsilon|t|^p,
\end{equation}
for all $(x,t)\in\R^N\times\R$.
\end{lemma}

\begin{proof}
Taking $\varepsilon>0$ and using $(f_1)$, we obtain $\delta>0$
such that
\begin{equation}\label{l11}
|f(x,t)|\leq\varepsilon |t|,\ x\in\R^N,\ |t|\leq\delta.
\end{equation}
By $(f_3)$ there exists $R>0$ satisfying
$$|f(x,t)|^\tau \leq a_1|t|^\tau\widehat{F}(x,t)\leq\frac{a_1}2|t|^{\tau+1}|f(x,t)|,\
x\in\R^N,\ |t|\geq R.$$ Then, setting $p=2\tau/(\tau-1)$, we can
use $\tau>N/2$ to conclude that $2<p<2^*$. Moreover,
\begin{equation}\label{l12}
|f(x,t)|\leq C|t|^{\frac{\tau+1}{\tau-1}}= C|t|^{p-1},\ x\in\R^N,\
|t|\geq R.
\end{equation}
From the continuity and periodicity of $f_0$ we obtain $M>0$ such
that
$$
|f_0(x,t)|\leq M,\ x\in\R^N,\ \delta\leq|t|\leq R.
$$
Now, using $(f_5)(ii)$ we get
$$
|f(x,t)|\leq
\|h\|_\infty|t|^{q-1}+M\leq\left(\|h\|_\infty+\frac{M}{\delta^{q-1}}\right)|t|^{q-1},\
x\in\R^N,\ \delta\leq|t|\leq R.
$$
This, \eqref{l11} and \eqref{l12} proves the first inequality in
\eqref{c1}. The second one follows directly by integration.
\end{proof}

In view of the above lemma it is well defined the
functional $I: H\to\R$ given by
$$
I(u)=\frac12\|u\|^2-\int F(x,u).
$$
Moreover, standard calculations show that $I \in
C^1(H,\mathbb{R})$ and the Gateaux derivative of $I$ has the
following form
$$
I'(u)v=\int (\nabla u\nabla v+V(x)uv)-\int f(x,u)v,
$$
for any $u,\,v \in H$. Hence, the critical points de $I$ are
precisely the weak solutions of the problem (\ref{eq1}).

In order to obtain the desired critical points we
shall use the following abstract result. We refer to \cite[Theorem
2.3]{lins2009quasilinear}.

\begin{theorem}[Local Mountain Pass Theorem]\label{lmpt}
Let $E$ be a real Banach space. Suppose that $I\in C^1(E,\R)$
satisfies $I(0)=0$ and
\begin{enumerate}
\item[$(I_1)$] there exist $\rho, \alpha>0$ such that $I(u)\geq\alpha>0$ for all $\|u\|=\rho$,
\item[$(I_2)$] there exist $e\in E$ with $\|e\|>\rho$ such that $I(e)\leq0$.
\end{enumerate}
If there exists $\gamma_0\in\Gamma=\{\gamma\in C([0,1], E):
\gamma(0)=0, \|\gamma(1)\|>\rho, I(\gamma(1))\leq0\}$ such that
$$
c=\max_{t\in[0,1]}I(\gamma_0(t))>0
$$
then $I$ possesses a nontrivial critical point
$u\in\gamma_0([0,1])$ at the level $c$.
\end{theorem}

In the next result we prove that the functional $I$ verifies the
geometric conditions of the Mountain Pass Theorem.

\begin{lemma}\label{geometry}
Suppose that $f$ satisfies $(f_1)$, $(f_3)$, $(f_4)$ and
$(f_5)(ii)$. Then $I$ satisfies $(I_1)$ and $(I_2)$.
\end{lemma}

\begin{proof} By Lemma \ref{lemma1.1} and Sobolev inequality we have
$$
\int F(x,u)\leq\varepsilon|u|_2^2+C_\varepsilon|u|_p^p\leq
c_1\varepsilon \|u\|^2+C\|u\|^p,
$$
for some $c_1>0$. Since $p>2$, we have
$$I(u)\geq\left(\frac12-c_1\varepsilon\right)\|u\|^2+o(\|u\|^2)\geq\alpha$$
for $\|u\|=\rho$ small enough. This proves $(I_1)$.

 In order to verify the condition $(I_2)$ we fix
$\varphi\in C^\infty_0(\R^N)$ satisfying $\varphi(x)\geq0$ in
$\R^N$ and $\|\varphi\|=1$. We claim that there is $R_0>0$ such
that, for any $R>R_0$, we have that $I(R\varphi)<0$. If this is
true it suffices to take $e=R \varphi$ with $R>0$ large enough to
get $(I_2)$.

For the proof of the claim we set $k=2/\int\varphi^2$ and use
$(f_4)$ to obtain $M>0$ satisfying
$$
F(x,t)\geq kt^2\ \text{for all}\ |t|\geq M.
$$
Hence, setting $A_R=\{x\in\R^N; \varphi(x)\geq M/R\}$, we get
\begin{equation} \label{desce}
\int F(x,R\varphi)\geq\int_{A_R} F(x,R\varphi)\geq k R^2
\int_{A_R} \varphi^2.
\end{equation}
Since $\varphi \geq 0$ we can choose $R_0>0$ such that, for any $R
\geq R_0$, it holds $ \int_{A_R}
\varphi^2\geq\frac12\int\varphi^2. $ It follows from the
definition of $k$ and \eqref{desce} that $\int F(x,R\varphi) \geq
R^2$ and therefore
$$
I(R\varphi)\leq\frac12R^2-R^2=-\frac12R^2<0,
$$
for any $R>R_0$.
\end{proof}

We recall that $I$ is said to satisfy the Cerami
condition at the level $c \in \mathbb{R}$ if any sequence $ (u_n)
\subseteq H$ such that
$$
\lim_{n\to+\infty} I(u_n) = c \quad \text{and} \quad
\lim_{n\to+\infty} (1 + \|u_n\|_E)\|I'(u_n)\|_{H'} = 0
$$
possesses a convergent subsequence in $H$. A sequence $(u_n)
\subseteq H$ as above is called Cerami sequence for $I$.

\begin{lemma}\label{bounded}
Suppose that $f$ satisfies $(f_1)-(f_4)$ and $(f_5)(ii)$. Then any
Cerami sequence for $I$ is bounded.
\end{lemma}

\begin{proof} We adapt here an argument from  ~\cite{ding2006multiple}. Let $(u_n)\subset H$ be
such that
$$
\lim_{n \to +\infty}I(u_n)=c \,\,\, \text{and}
\,\,\,\lim_{n\to+\infty}  (1+\|u_n\|)\|I'(u_n)\|_{H'}=0.
$$
It follows that
\begin{equation}\label{b1}
c+o_n(1)=I(u_n)-\frac12I'(u_n)u_n=\int\widehat{F}(x,u_n),
\end{equation}
where $o_n(1)$ stands for a quantity approaching zero as $n\to
+\infty$. Suppose by contradiction that, for some
subsequence still denote $(u_n)$, we have that $\|u_n\|\to\infty$.
By defining $v_n=\frac{u_n}{\|u_n\|}$ we obtain
$$
o_n(1)=\frac{I'(u_n)u_n}{\|u_n\|^2}=1-\int\frac{f(x,u_n)v_n}{\|u_n\|},
$$
and therefore
\begin{equation}\label{lim1}
\lim_{n \to +\infty} \int\frac{f(x,u_n)v_n}{\|u_n\|}=1.
\end{equation}

For any $r\geq0$ we set
$$
g(r)=\inf\{\widehat{F}(x,t); x\in\R^N, |t|\geq r\}.
$$
Let $R_1>0$ be given from $(f_3)$. For any
$|t|>R_1$, there holds
$$
a_1\widehat{F}(x,t)\geq\left(\frac{f(x,t)}{t}\right)^\tau\geq\left(\frac{2F(x,t)}{t^2}\right)^\tau.
$$
Hence, it follows from $(f_4)$ that $\widehat{F}(x,t)\to\infty$ as
$t\to\infty$ uniformly in $x\in\R^N$. This, $(f_2)$ and the
definition of $g$ imply that $g(r)>0$ for all $r>0$ and
$g(r)\to\infty$ as $r\to\infty$.

For $0\leq a<b$, we define
$$
\Omega_n(a,b)=\{x\in\R^N; a\leq|u_n(x)|<b\}
$$
and for $a>0$,
$$
c_a^b=\inf\left\{\frac{\widehat{F}(x,t)}{t^2}; x\in\R^N,
a\leq|t|\leq b\right\}.
$$
From $(f_2)$ we have that $c_a^b>0$. By using \eqref{b1}
and the above definitions we obtain
$$
\begin{array}{lcl}
c+o_n(1)&=&
\dint_{\Omega_n(0,a)}\widehat{F}(x,u_n)+\dint_{\Omega_n(a,b)}\widehat{F}(x,u_n)+\dint_{\Omega_n(b,\infty)}\widehat{F}(x,u_n) \vspace{0.2cm}\\
&\geq&
\dint_{\Omega_n(0,a)}\widehat{F}(x,u_n)+c_a^b\dint_{\Omega_n(a,b)}u_n^2+g(b)
|\Omega_n(b,\infty)|,
\end{array}
$$
and therefore, for some $C_1>0$, we have that
\begin{equation} \label{b2}
\max\left\{
\int_{\Omega_n(0,a)}\widehat{F}(x,u_n),\,c_a^b\int_{\Omega_n(a,b)}u_n^2,\,g(b)
|\Omega_n(b,\infty)|\right\} \leq C_1.
\end{equation}

The above inequality implies that $|\Omega_n(b,\infty)| \leq
C/g(b)$. Recalling that $g(b) \to +\infty$ as $b \to +\infty$ we
conclude that
\begin{equation} \label{limiteb}
\lim_{b \to +\infty}|\Omega_n(b,\infty)| = 0.
\end{equation}
Fixed $\mu\in[2,2^*)$ and $\nu\in(\mu,2^*)$,  by H\"older's
inequality and Sobolev embedding, we obtain, for some $C_2>0$,
\begin{align}
\int_{\Omega_n(b,\infty)}|v_n|^\mu&\leq\left(\int_{\Omega_n(b,\infty)}|v_n|^\nu\right)^{\mu/\nu}
|\Omega_n(b,\infty)|^{(\nu-\mu)/\nu}\nonumber\\
&\leq
C_2\|v_n\|^\mu|\Omega_n(b,\infty)|^{(\nu-\mu)/\nu}=C_2|\Omega_n(b,\infty)|^{(\nu-\mu)/\nu}\nonumber.
\end{align}
Since $\nu-\mu>0$ we conclude that
\begin{equation} \label{b3}
\lim_{b \to +\infty} \int_{\Omega_n(b,\infty)}|v_n|^\mu=0.
\end{equation}
Again from $(\ref{b2})$, for $0<a<b$ fixed, it follows that
$$
\int_{\Omega_n(a,b)}|v_n|^2=\frac1{\|u_n\|^2}\int_{\Omega_n(a,b)}u_n^2\leq
\frac1{\|u_n\|^2}\frac{C_1}{c_a^b} = o_n(1).
$$

Let $C_3>0$ be such that $|u|_2\leq C_3\|u\|$ for
all $u\in H$ and consider $\varepsilon\in(0,1/3)$. By $(f_1)$,
there exists $a_\varepsilon>0$ such that
$$
|f(x,u)|\leq\frac{\varepsilon|u|}{C_3^2}\ \text{for all}\ |u|\leq
a_\varepsilon.
$$
Hence,
\begin{align}\label{b4}
\int_{\Omega_n(0,a_\varepsilon)}\frac{f(x,u_n)v_n}{\|u_n\|}\leq\frac{\varepsilon}{C_3^2}\int_{\Omega_n(0,a_\varepsilon)}
v_n^2\leq\varepsilon.
\end{align}
Using $(f_5)$ and recalling that $h \in
L^{\infty}(\mathbb{R}^N,\mathbb{R})$ we obtain $C_4>0$ such that
$|f(x,u_n)|\leq C_4|u_n|$ for every
$x\in\Omega_n(a_\varepsilon,b_\varepsilon)$ and so,
\begin{align}\label{b5}
\int_{\Omega_n(a_\varepsilon,b_\varepsilon)}\frac{f(x,u_n)v_n}{\|u_n\|}\leq
C_4\int_{\Omega_n(a_\varepsilon,b_\varepsilon)}
v_n^2<\varepsilon,~~ \text{for all } n\geq n_0.
\end{align}

If we set $2\tau'=2\tau/(\tau-1)\in(2,2^*)$, we can use condition
$(f_3)$, $\eqref{b2}$ and H\"older's inequality to get
\begin{align}
\int_{\Omega_n(b_\varepsilon,\infty)}\frac{f(x,u_n)v_n}{\|u_n\|}&=\int_{\Omega_n(b_\varepsilon,\infty)}\frac{f(x,u_n)v_n^2}{|u_n|}\nonumber\\
&\leq\left(\int_{\Omega_n(b_\varepsilon,\infty)}\frac{|f(x,u_n)|^\tau}{|u_n|^\tau}\right)^{1/\tau}\left(\int_{\Omega_n(b_\varepsilon,\infty)}|v_n|^{2\tau'}\right)^{1/\tau'}\nonumber\\
&\leq a_1^{1/\tau}\left(\int_{\Omega_n(b_\varepsilon,\infty)}\widehat{F}(x,u_n)\right)^{1/\tau}\left(\int_{\Omega_n(b_\varepsilon,\infty)}|v_n|^{2\tau'}\right)^{1/\tau'}\nonumber\\
&\leq
C_1\left(\int_{\Omega_n(b_\varepsilon,\infty)}|v_n|^{2\tau'}\right)^{1/\tau'}
\nonumber.
\end{align}
This expression and \eqref{b3} provides $b_{\varepsilon}>0$ large
in such way that
\begin{equation}\label{b6}
\int_{\Omega_n(b_\varepsilon,\infty)}\frac{f(x,u_n)v_n}{\|u_n\|}<\varepsilon,\,\,\,\mbox{for
all } n\geq n_0.
\end{equation}

Finally, the estimates $(\ref{b4})-(\ref{b6})$ imply
$$
\int\frac{f(x,u_n)v_n}{\|u_n\|}\leq 3\varepsilon<1,
$$
which contradicts $(\ref{lim1})$. Therefore $(u_n)$ is bounded in
$H$.

\end{proof}

\begin{remark}\label{remark2}
If $f$ is periodic we can obtain the estimate in
\eqref{b5} without the condition $(f_5)$. Moreover, in this case,
it follows from periodicity and continuity of $F_0$ that
$\frac{F_0(x,u)}{u^2}\geq k=k(a,b)>0$ for all $x\in\Omega_n(a,b)$.
Of course $c_a^b\geq k>0$ and therefore the above lemma holds
under the setting of Theorem \ref{th2}.
\end{remark}

\begin{lemma}\label{nonvanishing}
Suppose that $f$ satisfies $(f_1)$ and $(f_2)$. Let $(u_n)\subset
H$ be a Cerami sequence for $I$ at level $c>0$. If
$u_n\rightharpoonup0$ weakly in $H$ then there exist a sequence
$(y_n)\subset\R^N$ and $R>0$, $\alpha>0$ such that
$|y_n|\to\infty$ and
$$
\limsup_{n\to\infty}\int_{B_R(y_n)}|u_n|^2\geq\alpha>0
$$
\end{lemma}

\begin{proof}
Suppose, by contradiction, that the lemma is false. Then, for any
$R>0$, we have that
$$
\limsup_{n\to\infty}\int_{B_R(y)}|u_n|^2=0 \ \text{for all}\ R>0.
$$
Hence, we can use a result of Lions (see
~\cite{lions1984concentration}) to conclude that $|u_n|_s \to 0$
for any $s\in(2,2^*)$. It follows from the second inequality in
\eqref{c1} that
$$
\limsup_{n\to+\infty} \int F(x,u_n)\leq \limsup_{n\to \infty}
\left( \varepsilon \int |u_n|^2 + C_{\varepsilon} \int |u_n|^p
\right)\leq C \varepsilon,
$$
where we have used the boundedness of $(u_n)$ in
$L^2(\mathbb{R}^N)$. Since $\varepsilon$ is arbitrary we conclude
that $\int F(x,u_n) \to 0$ as $n\to+\infty$. The same argument and
the first inequality in \eqref{c1} imply that $\int f(x,u_n)u_n
\to 0$ as $n\to+\infty$.

Since $(u_n)$ is a Cerami sequence, we get
$$
c=\lim_{n\to\infty}\left[I(u_n)-\frac12 I'(u_n)u_n\right]
=\lim_{n\to\infty}\int(\frac12f(x,u_n)u_n-F(x,u_n))=0
$$
which contradicts $c>0$. The lemma is proved.
\end{proof}

We finish the section by stating two technical
convergence results. The proofs can be found in \cite[Lemmas 5.1
and 5.2]{lins2009quasilinear}, respectively.

\begin{lemma}\label{l2.4}
Suppose that $(V)$ and $(f_5)$ are satisfied. Let $(u_n)\subset H$
be a bounded sequence and $v_n(x)=v(x-y_n)$, where $v\in H$ and
$(y_n)\subset\R^N$. If $|y_n|\to\infty$, then we have
$$
[V_0(x)-V(x)]u_nv_n\to0,
$$
$$
[f_0(x,u_n)-f(x,u_n)]v_n\to0,
$$
strongly in $L^1(\R^N)$, as $n\to\infty$.
\end{lemma}

\begin{lemma}
\label{l2.5} Suppose $h\in\mathcal{F}$ and $s\in[2,2^*]$. If
$(v_n) \subseteq H^1(\mathbb{R}^N)$ is such that
$v_n\rightharpoonup v$ weakly in $H$, then
$$
\lim_{n\to +\infty}\int h|v_n|^s  = \int h|v|^s.
$$
\end{lemma}

\section{Proofs of the main results}

In section, we denote by $I_0:H\to\R$ the functional associated
with the periodic problem, namely
$$
I_0(u)=\frac12\int(|\nabla u|^2+V_0(x)u^2)-\int F_0(x,u).
$$
We also consider the following norm in $H^1(\mathbb{R}^N)$
$$
\|u\|_0=\left(\int(|\nabla u|^2+V_0(x)u^2\right)^2,
$$
which is equivalent to the usual norm of this space.

We are ready to prove our main theorem as follows:

\vspace{0.2cm}

\noindent \textit{Proof of Theorem \ref{th1}.} By Lemma
\ref{geometry} and the Mountain Pass Theorem there exists a
sequence $(u_n)\subset H$ such that
\begin{equation} \label{eps}
I(u_n)\to c\geq\alpha>0\ \text{and}\ (1+\|u_n\|)I'(u_n)\to0,\ \text{as}\ n\to\infty.
\end{equation}
Applying Lemma \ref{bounded}, we may assume, without loss
generality, that $u_n \rightharpoonup u$ weakly in $H$. We claim
that $I'(u)=0$. Indeed, since $C^\infty_0(\R^N)$ is dense in $H$,
it suffices to show that $I'(u)\varphi=0$ for all $\varphi\in
C^\infty_0(\R^N)$. We have
\begin{align}\label{d1}
I'(u_n)\varphi-I'(u)\varphi=o_n(1)-\int[f(x,u_n)-f(x,u)]\varphi.
\end{align}
Using the Sobolev embedding theorem we can assume that, up to a
subsequence,  $u_n\to u$ in $L^s_{loc}(\R^N)$ for each
$s\in[1,2^*)$ and
\begin{align*}
&u_n(x)\to u(x)\ \text{a.e. on}\ K, \ \text{as}\ n\to\infty,\\
&|u_n(x)|\leq w_s(x)\in L^s(K),\ \text{for every}\ n\in\mathbb{N}\ \text{and a.e. on}\ K,
\end{align*}
where $K$ denotes the support of the function $\varphi$.
Therefore,
$$
f(x,u_n)\to f(x,u)\ \text{a.e. on}\ K,\ \text{as}\ n\to\infty,
$$
and using (\ref{c1}), we get
$$
|f(x,u_n)\varphi|\leq\varepsilon|w_2||\varphi|+C_\varepsilon|w_{p-1}||\varphi|\in
L^1(K).
$$
Thus, taking the limit in (\ref{d1}) and using the Lebesgue
Dominated Convergence we get
$$
I'(u)\varphi=\lim_{n\to\infty}I'(u_n)\varphi=0,
$$
which implies $I'(u)=0$.

If $u\neq0$, the theorem is proved. So, we deal in the sequel with
the case $u=0$. By Lemma \ref{nonvanishing}, we recall that there
exist a sequence $(y_n)\subset\R^N$, $R>0$, and $\alpha>0$ such
that $|y_n|\to\infty$ as $n\to\infty$, and
\begin{equation}\label{d2}
\limsup_{n\to\infty}\int_{B_R(y_n)}|u_n|^2\geq\alpha>0.
\end{equation}
Without loss of generality we may assume that
$(y_n)\subset\mathbb{Z}^N$ (see ~\cite[page
7]{chabrowski1999weak}). Writing  $\widetilde{u}_n(x)=u_n(x+y_n)$
and observing that $\|\widetilde{u}_n\|=\|u_n\|_0$, up to
subsequence we have $\widetilde{u}_n\rightharpoonup\widetilde{u}$
in $H$, $\widetilde{u}_n\to\widetilde{u}$ in $L^2_{loc}(\R^N)$ and
for almost every $x\in\R^N$. From $(\ref{d2})$, we have
$\widetilde{u}\neq0$.

\begin{claim} $I'_0(\widetilde{u})=0$
\end{claim}

To prove the claim we take $\varphi\in C^\infty_0(\R^N)$ and
define, for each $n \in\mathbb{N}$, $\varphi_n(x)=\varphi(x-y_n)$.
Arguing as in the beginning of the proof and using the periodicity
of $f_0$ we get
$$
I_0'(\widetilde{u}) \varphi = I_0'(\widetilde{u}_n) \varphi +
o_n(1) = I_0'(u_n) \varphi_n + o_n(1),
$$
and therefore it suffices to check that $I_0'(u_n)
\varphi_n=o_n(1)$. To achieve this objective we notice that, by
Lemma $\ref{l2.4}$,
$$
\begin{array}{lcl}
I'_0(u_n)\varphi_n &=& I'(u_n)\varphi_n
+\displaystyle\int[V_0(x)-V(x)]u_n\varphi_n-\displaystyle\int[f_0(x,u_n)-f(x,u)]\varphi_n
\vspace{0.2cm} \\ & = & I'(u_n)\varphi_n + o_n(1).
\end{array}
$$
So, by (\ref{eps}), the claim is verified.

\begin{claim}
$\liminf\limits_{n\to\infty}\displaystyle\int\widehat{F}(x,u_n)\geq\displaystyle\int\widehat{F}_0(x,\widetilde{u})$
\end{claim}

By using $(f_5)(ii)$ and a straightforward calculation we
obtain
$$
|\widehat{F}(x,t)-\widehat{F}_0(x,t)|\leq\left(\frac12+\frac{1}{q}\right)h(x)|t|^q.
$$
Since $u_n \rightharpoonup 0$ weakly in $H$, it follows from the
above inequality and Lemma \ref{l2.5} that
\begin{align*}
\lim_{n\to\infty}\int\widehat{F}(x,u_n)&=\lim_{n\to\infty}\int\widehat{F}_0(x,u_n)\\
&=\liminf_{n\to\infty}\int\widehat{F}_0(x,\widetilde{u}_n)\geq
\int\widehat{F}_0(x,\widetilde{u}),
\end{align*}
where we also have used the periodicity of $\widehat{F}_0$.

By using \eqref{eps} and the above claim we get
\begin{align*}
c&=\lim_{n\to\infty}[I(u_n)-\frac12 I'(u_n)u_n]=\liminf_{n\to\infty}\int\widehat{F}(x,u_n)\\
&\geq
\int\widehat{F}_0(x,\widetilde{u})=I_0(\widetilde{u})-\frac12
I_0'(\widetilde{u})\widetilde{u}=I_0(\widetilde{u}),
\end{align*}
and therefore $I_0(\widetilde{u})\leq c$. It follows from
$(f_5)(iii)$ that
$\max_{t\geq0}I_0(t\widetilde{u})=I_0(\widetilde{u})$. Hence, by
the definition of $c$, $(V)$ and $(f_5)(i)$, we have that
$$
c\leq\max_{t\geq0}I(t\widetilde{u})\leq\max_{t\geq0}I_0(t\widetilde{u})=I_0(\widetilde{u})\leq c
$$
We can now invoke  Theorem \ref{lmpt} to conclude that $I$
possesses a critical point at level $c>0$. This finishes the
proof. \hfill $\Box$ \vspace{0.2cm}

We proceed now with the proof of the periodic result.

\vspace{0.2cm} \noindent \textit{Proof of Theorem $\ref{th2}$.} We
first notice that Lemmas $2.1$, $2.2$ and $2.3$ are still valid
under the assumptions of the Theorem $\ref{th2}$. Hence, by Lemma
\ref{geometry}, we obtain a sequence $(u_n)\subset H$ such that
$$
\lim_{n\to+\infty}I_0(u_n) = c_0 \quad \text{and} \quad
\lim_{n\to+\infty} (1 + \|u_n\|_0)\|I_0'(u_n)\| = 0,
$$
where $c_0$ is the mountain-pass level of $I_0$. Arguing as in the
proof of Theorem \ref{th1} we conclude that $u_n\rightharpoonup u$
weakly in $H$ with $I_0'(u)=0$.

As before, we need only consider the case $u=0$. By the Lemma
$\ref{nonvanishing}$, there is a sequence
$(y_n)\subset\mathbb{Z}^N$ (see \cite[page 7]{chabrowski1999weak}), $R>0$ and
$\alpha>0$ such that $|y_n|\to\infty$ as $n\to\infty$ and
\begin{equation}\label{**}
\limsup_{n\to\infty}\int_{B_R(y_n)}|u_n|^2\geq\alpha>0.
\end{equation}
Writing $\widetilde{u}_n(x)=u_n(x+y_n)$ and observing that
$\|\widetilde{u}_n\|_0=\|u_n\|_0$, up  to subsequence, we have
$\widetilde{u}_n\rightharpoonup\widetilde{u}$ weakly in $H$,
$\widetilde{u}_n\to\widetilde{u}$ in $L^2_{loc}(\R^N)$ and
$\widetilde{u}_n(x)\to\widetilde{u}(x)$ almost everywhere in
$\R^N$. The local convergence and \eqref{**} imply that
$\widetilde{u}\neq0$. Arguing as in Claim 1 of the proof of
Theorem \ref{th1} we conclude that $I'_0(\widetilde{u})=0$ and
therefore we obtain a nonzero weak solution.

In view of the above existence result it is well defined
$$
m=\inf\{I_0(u); u\in E \ \mbox{and}\ I'(u)=0\}>0.
$$
We claim that $m$ is achieved. Indeed, let $(u_n) \subset H$ be a
minimizing sequence for $m$, namely
$$
I_0(u_n)\to m,  \ I_0'(u_n)=0\ \mbox{and}\ u_n\neq0.
$$
Since $(u_n)$ is a Cerami sequence for $I_0$ it follows from Lemma
\ref{bounded} that it is bounded. Moreover, using $I'_0(u_n)u_n=0$
and (\ref{c1}) with $\varepsilon$ small, we can obtain $k>0$
satisfying $\|u_n\|_0\geq k$. Thus, arguing as in the preceding
paragraph, we obtain a translated subsequence $(\widetilde{u}_n)$
which has a nonzero weak limit $u_0$ such that $I'_0(u_0)=0$ and
$\widetilde{u}_n(x)\to u_0(x)$ a.e. in $\R^N$. By Fatou's lemma,
\begin{align*}
m&=\lim_{n\to\infty} I_0(u_n)=\lim_{n\to\infty} I_0(\widetilde{u}_n)\\
&=\liminf_{n\to\infty} \int\widehat{F}_0(x,\widetilde{u}_n)\\
&\geq\int\widehat{F}_0(x,u_0)=I_0(u_0).
\end{align*}
Consequently $I_0(u_0)=m$ and therefore $u_0 \neq 0$ is a ground
state solution.  \hfill $\Box$

\section*{Acknowledgments}

This is part of the author?s Ph.D. thesis, written under the supervision
of Dr. Marcelo F. Furtado at the University of Bras\'ilia. 
The author would like to express his sincere appreciation to his advisor
for his guidance and advice throughout this research.


 \end{document}